\documentclass{amsart}

\usepackage{graphicx}
\usepackage{amsmath,amscd}
\usepackage[all]{xy}
\usepackage{hyperref}
\usepackage[alphabetic,lite]{amsrefs}
\usepackage{listings}
\usepackage{color}
\definecolor{dkgreen}{rgb}{0,0.6,0}
\definecolor{gray}{rgb}{0.5,0.5,0.5}
\definecolor{mauve}{rgb}{0.58,0,0.82}

\newtheorem{theorem}{Theorem}[section]

\newtheorem{proposition}[theorem]{Proposition}
\newtheorem{corollary}[theorem]{Corollary}

\theoremstyle{definition}
\newtheorem{definition}[theorem]{Definition}

\theoremstyle{remark}
\newtheorem{remark}[theorem]{Remark}

\numberwithin{equation}{section}

\newcommand{\Image}{\mathrm{Image}\,}
\newcommand{\Pic}{\mathrm{Pic}\,}
\newcommand{\PGL}{\mathrm{PGL}\,}

\newcommand{\length}{{\rm length }}
\newcommand{\PP}{{\mathbb P }}
\newcommand{\GG}{{\mathbb G }}
\newcommand{\sO}{{\mathcal O }}
\newcommand{\sI}{{\mathcal I }}
\newcommand{\sU}{{\mathcal U }}
\newcommand{\sM}{{\mathcal M }}
\newcommand{\sW}{{\mathcal W }}
\newcommand{\sH}{{\mathcal H }}
\newcommand{\Hilb}{{\rm Hilb }}

\begin{document}

\title{Unirationality of the Hurwitz space $\mathcal{H}_{9,8}$}

\author{Hamid Damadi}
\address{Department of Pure Mathematics
Faculty of Mathematics and Computer Science
Amirkabir University of Technology (Tehran Polytechnic)
424, Hafez Ave., Tehran 15914, Iran.}
\email{Hamid.Damadi@aut.ac.ir}

\author{Frank-Olaf Schreyer}
\address{Mathematik und Informatik, Universit\"at des Saarlandes, Campus E2 4, D-66123 Saarbr\"ucken, Germany}
\email{schreyer@math.uni-sb.de}

\subjclass[2010]{14H10, 14H50, 14H51, 14M20, 14Q05}

\date{\today}

\dedicatory{}

\keywords{Hurwitz space, Unirationality}

\fontsize{10pt}{14pt}
\begin{abstract}
In this paper we prove that the Hurwitz space $\mathcal{H}_{9,8}$, which parameterizes  8-sheeted covers of $\PP^1$ by  curves of genus 9, is unirational. Our construction leads to an explicit 
Macaulay2 code, which will randomly produce  a nodal curve of degree 8 of geometric genus 9 with 12 double points and together with a pencil of degree 8.
\end{abstract}
\maketitle

\section{Introduction}
A Hurwitz space parameterizes "maps of curves to $\mathbb{P}^1$":
$$
\mathcal{H}_{g,d}=
\left\{ f \colon C \to \PP^1 \mid f  \hbox{ a simply branched cover of degree } d \hbox{ and } C \hbox{ of genus } g \right\}
$$
where simply branched  means that for every ramification point $p \in C$ the ramification index
$$e_p = \length(\Omega_{C/\PP^1})_ p + 1 = 2,$$
and no two ramification points lie over the same point of
$\mathbb{P}^1$. In particular, $f$ is ramified in $w=2g-2+2d$ distinct points by the Riemann-Hurwitz formula.\\
Clebsch in \cite{Cle} showed that $\mathcal{H}_{g,d}$ is a
smooth and connected hence irreducible quasi-projective variety. Recall the following diagram from \cite{ACGH}
$$
\xymatrix{
\mathcal{H}_{g,d}\ar[d] \ar[r]^{\pi}  & \mathcal{M}_g  \\
\mathrm{Sym}^w (\mathbb{P}^1)\backslash \Delta=\mathbb{P}^w\backslash \Delta&
}
$$
where $\Delta$ denotes the closed subscheme of the points in $\mathrm{Sym}^w (\mathbb{P}^1)$ with at least two identical summands.\\
The map $\pi$ is a natural forgetful map which is dominant if and only if $d\geq \frac{g+2}{2}$ \cite{ACGH}.
The downward map is described by mapping a branched
cover to its branch divisor which consists of an unordered $w$-tuple of distinct points. Given such a $w$-tuple,
recovering the curve amounts to adding monodromy data, i.e., information how the sheets of the covering
glue as we go around the branch points. It follows that $\mathcal{H} _{g,d}\longrightarrow \mathbb{P}^w \backslash \Delta$ is a finite covering, hence
$$
\dim \mathcal{H}_{g,d} = \dim (\mathbb{P}^w \backslash \Delta)  = w=2g-2+2d.
$$

The book \cite{HM} and the paper \cite{F} are excellent references for the study of Hurwitz 
spaces. \medskip
\begin{definition}
A variety $V$ is unirational if there are some projective spaces $\mathbb{P}^n$ and a dominant rational map
$$
\mathbb{P}^n\dashrightarrow V.
$$
\end{definition}
The unirationality of a moduli
space leads to a parametrization of a dominant family of the objects in terms of independent  parameters.\\
The unirationality of $\mathcal{H}_{g,d}$ for $2 \leq d \leq 5$ and arbitrary $g\geq 2$ is known for a long time. The case $d = 5$ is based on the work of Petri in \cite{P}. In \cite{AC} Arbarello and Cornalba proved that in the following cases the space $\mathcal{H}_{g,d}$ is unirational:
$$
\left\{
\begin{array}{cclcc}
d\leq 5 & and & g \geq d+1 &&\\
d=6 & and & 5\leq g \leq 10 & or & g=12\\
d=7 & and & g=7 &&
\end{array}
\right.
$$
Geiss in \cite{Ge} and \cite{Ge12} showed that $\mathcal{H}_{g,6}$  for $g \leq 28$ and $g = 30, 31, 35, 36, 40, 45$ and $\mathcal{H}_{g,7}$ for $6\leq g\leq 12$ are unirational.

In a sequence of papers Mumford and Harris \cite{HM82}, Harris \cite{H84}, Eisenbud and Harris \cite{EH} and Farkas \cite{Far} showed that the moduli spaces ${\mathcal{M
}}_g$ of curves of genus $g$ are of general type for $g = 22$ or $g \geq 24$, and of positive Kodaira dimension for $g=23$.
This implies that
$$
{\mathcal{H}}_{g,d} \hbox{ is not unirational for } g \geq 22 \hbox{ and } d \geq \lfloor\frac{(g + 2)}{2}\rfloor,
 $$
 since $\sH_{g,d}$ dominates $\sM_g$ for $d \geq \lfloor\frac{(g + 2)}{2}\rfloor$.

For small genus
Mukai proved the following theorem:
\begin{theorem}[Mukai,\cite{M1},\cite{M2}]
A general canonical curve $C$ of genus $g = 7, 8, 9$ arises as
transversal intersection of a linear space with a homogeneous
variety:
$$
\left\{
\begin{array}{cll}
g=7 & C = \mathbb{P}^6 \cap \mathrm{Spinor} 10 \subset \mathbb{P}^{15}& \mbox{Isotropic subspaces of}\ \ Q^8 \subset \mathbb{P}^9\\
g=8 & C = \mathbb{P}^7 \cap \mathbb{G}(2, 6)^8 \subset \mathbb{P}^{14}& \mbox{Grassmannian of lines in}\ \mathbb{P}^5\\
g=9& C = \mathbb{P}^8 \cap \mathbb{L}(3, 6)^6 \subset \mathbb{P}^{13}& \mbox{Lagrangian subspaces of} \ (\mathbb{C}^6 , \omega)
\end{array}
\right.
$$
\end{theorem}
So the moduli spaces $\mathcal{M}_{g,g}$ of $g$-pointed curves of genus $g$ and the universal Picard varieties $\Pic^d_g \longrightarrow\mathcal{M}_g$ are unirational for $g \leq 9$. Then as a corollary the Hurwitz spaces $\mathcal{H}_{g,d}$ are unirational for $g\leq 9$ and $d\geq g$. This means in the range $d \leq 5$ or $g\leq 9$ only the case $\mathcal{H}_{9,8}$ is open.\\

The idea of this paper is to deduce the uniratioality of the Hurwitz space $\mathcal{H}_{9,8}$ from that of the Brill-Noether locus $\sW^1_{9,8}$, whose definition we recall in Section 2 and whose unirationality we prove in Section 3. \\

Part of the proof is based on explicit computations in Macaulay2.
All necessary computations can be found in 
\href{http://www.math.uni-sb.de/ag/schreyer/index.php/computeralgebra}{\color{magenta} RandomCurveOfGenus9\-WithAPencilOfDegree8}
a $\mathtt{Macaulay2}$ package \cite{DS} available online. 
We choose a finite prime field $\mathbb F_p$ in our code $p=10007$. We then construct an example along steps (1) to (4) of Section 3 below
 and verify that the example
has all desired open properties  specified in Section 3 over $\mathbb F_p$. Since we may regard all computations over 
the field $\mathbb{F}_p={\mathbb{Z}}/{(p)}$ as the reduction mod $p$ of an example with integer coefficients, semi-continuity proves that
the corresponding example over $\mathbb Q$ has the same desired properties. Finally, from the existence of an example over $\mathbb Q$ for which the  desired properties hold, we will deduce  in Section 3 the unirationality  of $\mathcal{H}_{9,8}$ over $\mathbb Q$ and hence over $\mathbb C$.

%
\section{Preliminaries and Notations}
The notations and definitions are from two books \cite{ACGH} and \cite{ACG}. For 
reader's convenience we recall some of them. Let $C$ be a smooth curve of genus g and $\Pic^d(C)$ be the Picard variety consisting of isomorphism classes of line bundles of degree $d$ on $C$. 
Let $L$ be a point of the Brill-Noether locus
$$
W^r_d(C)=\{L\in \Pic^d(C)|\ \ h^0(C,L)\geq r+1\}
$$
not belonging to $W^{r+1}_d(C)$. The tangent space to $W^r_d(C)$ at $L$ is
$$
T_L(W^r_d(C)) = (\Image \mu_L)^{\perp},
$$
where
$$
\mu_L:H^0(C,L)\otimes H^0(C,\omega_C\otimes L^{-1})\longrightarrow H^0(C,\omega_C)
$$
is the Petri map. The scheme $W^r_d(C)$ is smooth at $L$ of dimension equal to the Brill-Noether number
 $$
 \rho=\rho(g,d,r)=g-(r+1)(g-d+r)\geq 0 \leqno{(*)}
$$
if and only if $\mu_L$ is injective. This is the case for a general curve $C$ and every $L \in W^r_d(C) \setminus W^{r+1}_d(C)$ by the Gieseker-Petri Theorem.

 The diagram from the introduction refines to
$$
\xymatrix{
\mathcal{H}_{g,d}\ar[d] \ar[r]^{\alpha}  & {\mathcal G}^1_{g,d} \ar[r]^\beta& {\mathcal W}^1_{g,d}\ar[r]^\gamma& \mathcal{M}_g  \\
\mathrm{Sym}^w (\mathbb{P}^1)\backslash \Delta& &
}
$$
where 
 $$
 {\mathcal W}^1_{g,d}=\{ (C,L) \mid C \in {\mathcal M}_g, L \in \Pic^d(C) \hbox{ with } h^0(C,L) \ge 2 \}
 $$
 is  the universal Brill-Noether locus and
$$
{\mathcal G}^1_{g,d}=\{ (C,L,V) \mid (C,L) \in {\mathcal W}^1_{g,d}, \, V \subset H^0(C,L) \hbox{ a 2-dimensional subspace} \}
$$
is the universal space of pencils of degree $d$.
The fiber of the map $\alpha$ over a  general point  $(C,L,V)$ is $\PGL(V)$, the fiber of $\beta$ over a general point $(C,L)$ is a
Grassmannian $\GG(2,H^0(C,L))$. Finally the map $\gamma$ is dominant, if and only if $d\geq \frac{g+2}{2}$, and in this case the fiber
of $\gamma\circ \beta $ over a general point $C \in {\mathcal M_g}$ is the space 
$G^1_d(C)$ of dimension $\rho=g-2(g-d+1)$ of pencils of degree $d$ on $C$. Hence
$$
\dim\mathcal{G}^1_{g,d} = w -3 = 2g + 2d - 5=3g-3+g-2(g-d+1)
$$
verifies Riemann's count for $\dim {\mathcal M_g}=3g-3$.

\begin{proposition}[\cite{ACG} Proposition 6.8]\label{prop}
If $g > 1$, $d \geq 2$, and $d \leq g + 1$, then the singular locus of $\mathcal{W}^1_{g,d}$ is $\mathcal{W}^2_{g,d}$, and
$$
\dim \mathcal{W}^1_{g,d} = 3g - 3 + \rho = 2g + 2d - 5 .
$$
\end{proposition}
We will also need the Severi variety $\mathcal{U}_{g,d}$, which parameterizes reduced and irreducible plane curves of degree $d$ and geometric genus $g$ having only $\delta$ nodes as singularities (More details about Severi varieties can be found in \cite{H} and \cite{HM}). Consider the following diagram:
$$
\xymatrix{
\mathcal{U}_{g,d}\ar[d] \ar[r]&  \mathcal{M}_g  \\
\mathrm{Sym}^{\delta} (\mathbb{P}^2)\backslash \Delta&
}
$$
A single point imposes 3 conditions on the linear system of plane curves of degree $d$ to become a node. In case that general
$\delta$ points impose $3\delta$ linearly independent conditions to be nodes, we have that $\mathcal{U}_{g,d} \to \mathrm{Sym}^{\delta} (\mathbb{P}^2)$ is dominant and
$$
\dim \ \mathcal{U}_{g,d}=\frac{d(d+3)}{2}-3\delta+2\delta = 3d+g-1.
$$

\section{Construction}\label{secc}
Let $C$ be a general curve of genus $g=9$. The Brill-Noether number $\rho(9,8,2)=0$ by $(*)$ hence $C$ has finitely many $g^2_8$'s, ie. finitely many linear system of divisors of degree $8$ and dimension $2$. We pick one of these linear systems and consider the corresponding plane model $\Gamma \subset \PP^2$. Then $\Gamma$ will be a curve of degree $d=8$ with
$$ \delta={d-1 \choose 2} -g= 21-9 =12$$
nodes by the degree-genus formula. To see that all nodes are indeed ordinary nodes, it suffices to exhibit an example, which nowadays is easily verified by computer algebra: For a  general set $P=\{p_1,\ldots,p_{12} \} \subset \PP^2$ of 12 points we expect for the ideal sheaf $\sI_P$ of the zero-dimensional schem $P$ that
$$ h^0(\PP^2, \mathcal I^2_P(8))=9= {10 \choose 2}-3\cdot12$$
holds, since imposing a node are 3 linear conditions on the coefficients of the equation. An example computation shows that this is indeed the case  and that furthermore, a general form in $H^0(\PP^2, \mathcal I_P^2(8))$ defines a curve $\Gamma$ with 12 ordinary nodes as its sole singularities.
Moreover, by  \cite[Theorem 1.1]{EU}, $\Gamma$ is irreducible, since the resolution of twelve general points has the shape
$$0 \leftarrow \sO_P \leftarrow \sO \leftarrow \sO(-4)^3 \leftarrow \sO(-6)^2\leftarrow 0.$$
In particular, the syzygy module of the homogeneous ideal $I_P$ has no generators  in degree $8$. \medskip

A key point in our construction is the observation that we can find 12-nodal octics which pass through 8 further points $q_1,\ldots,q_8$, since  $H^0(\PP^2, \mathcal I^2_P(8))$ is 9-dimensional. \medskip

Blowing-up the double points yields a diagram
$$
\xymatrix{
C=\tilde{\Gamma} \ar[r] \ar[d]& S=Bl_{p_1,\ldots,p_{12}} \mathbb{P}^2 \ar[d]^v\\
\Gamma  \ar[r]&        \mathbb{P}^2
}
$$
and $C$  is isomorphic to the proper transform  $\tilde{\Gamma}$ of  $\Gamma$ in $S$. Denote by
\begin{itemize}
\item
$E_i$ = class of the exceptional divisor over the point $p_i$,
\item
$l = v^{\ast} h$, where $h$ is the class of a line in $\mathbb{P}^2$, and
\item
$E=\sum^{12}_{j=1} E_i$.
\end{itemize}
Then
\begin{itemize}
\item
$C \sim 8l-2E$,
\item
$\omega_S \sim -3l+E$ is the canonical divisor on $S$, and
\item
$\omega_C \sim (5l-E)|_C$
\end{itemize}
by the adjunction formula.  From $H^1(\omega_S)=0$ and  the exact sequence
$$
0\longrightarrow \mathcal{O}_S(-3l+E)\longrightarrow \mathcal{O}_S(5l-E)\longrightarrow \omega_C\longrightarrow 0
$$
we conclude  that $|\omega_C|$ is cut out on $C \subset S$ by the strict transforms of quintics in $\mathbb{P}^2$ through $p_1,\ldots,p_{12}$.
We work with $\Gamma$ as a model of $C$ which allows us to understand the possible $g^1_8$'s on $C$.

Let $|D|=g^1_8$ be  a complete base point free pencil of divisors of degree $8$ on $C$ corresponding to a general point
$L=\sO(D) \in W^1_8(C)$.  By Riemann-Roch
$|\omega_C(-D)|$ is another pencil of degree $8$, which is again base point free, since $L$ is general in $W^1_8(C)$.
A general divisor $q_1+\ldots+q_8 \in |\omega_C(-D)|$ has $8$ distinct points as its support, which are disjoint  from the nodes of $\Gamma$. Since
$D+q_1+\ldots+q_8 \in |\omega_C|$, 
we conclude  that the pencil $|D|=|\omega_C(-\sum^{8}_{i=1}q_i)|$ is cut out on $\Gamma$  by a pencil $|F|$ of quintics through $p_1,\ldots,p_{12},q_1,\ldots,q_8$.

Thus for $P=\{p_1,\ldots,p_{12}\}$ and $Q=\{q_1,\ldots,q_8\}$ we have
$$
h^0(\mathbb{P}^2, \mathcal{I}_{P\cup Q}(5))=2>1=h^0(\mathbb{P}^2,\mathcal{O}_{\mathbb{P}^2}(5))-20,
$$
i.e., the points $p_1,\ldots,p_{12},q_1,\ldots,q_8$ fail to impose independent conditions on quintics.

By Bezout's Theorem two quintics intersect in $25$ points counted with multiplicity, unless they have a common component.
So we expect that the pencil $|F|$ has $5$ further base points $r_1,\ldots,r_5 \subset \PP^2$ away from $\Gamma$. \medskip

The basic idea of the construction is to reverse the order in which we choose the data:

\begin{enumerate}\label{construction}
\item Choose random collections $P=\{p_1,\ldots,p_{12}\}$ and $R=\{r_1,\ldots,r_5\}$ of $12$ and $5$ points in $\PP^2$ with $P \cap R =\emptyset$.
\item Randomly choose a pencil $\langle f_1,f_2\rangle \in \GG(2, H^0(\PP^2,\sI_{P \cup R}(5)))$ of quintic forms through $P\cup R$.
\item Define $Q=\{q_1,\ldots,q_8\} \subset \PP^2$ as the residual intersection given by the homogeneous ideal $I_Q=(f_1, f_2): I_{P\cup R}$.
\item Choose a section in $H^0(\PP^2,\sI_P^2\cap \sI_Q(8))$ and take $\Gamma$ to be its zero loci.
\end{enumerate}
Then for this specific example and hence for general choices, we verify that the normalization $C$ of $\Gamma$ is a smooth curve of genus $g=9$, and that $C \to \PP^1, x \mapsto (f_1(x):f_2(x))$
defines a degree  $8$ simply branched cover of $\PP^1$. \medskip

So on one hand we have the open subvariety
$$
X_1=\{ (P,R)  \in \Hilb_{12}(\PP^2) \times \Hilb_5(\PP^2) \mid P, R \hbox{ reduced with } h^0(\sI_{P\cup R}(5))=4 \}
$$
of the Hilbert scheme of $12+5$ points in $\PP^2$, since the condition $h^0(\sI_{P\cup R}(5))=4$ implies that $P$ and $R$ are disjoint,
and  the $\GG(2,4)$-bundle over $X_1$ of triples
$$X_2=\{ (P,R,|F|) \mid (P,R) \in X_1,\ |F| \in \GG(2,H^0(\PP^2,\sI_{P\cup R}(5))) \}.$$
On the other hand we have
$$Y_1= \sU_{9,8}\times_{\sM_9} \sW_{9,8}^1 $$
and the $\PP^1$-bundle of triples
$$Y_2 = \{(\Gamma,L,Q) \mid \Gamma \in \sU_{9,8}, L \in W^1_8(C), Q=q_1+\ldots +q_8 \in |\omega_C \otimes L^{-1}| \},
$$
where $C$ denotes the normalization of $\Gamma$.

\begin{theorem} The construction above defines a birational map
$X_2 \dasharrow Y_2$.
\end{theorem}

\begin{proof} Since
\begin{align*}
\dim X_2 &= \dim \Hilb_{12+5}(\PP^2)+ \dim \GG(2,4) \cr
&= 34+4=38
\end{align*}
and
\begin{align*}
\dim Y_2  &= \dim \sU_{9,8}+\dim W^1_8(C) + \dim |\omega_C\otimes L^{-1}| \\
&= 32+5+1=38
\end{align*}
these two spaces have the same dimension. However to prove the theorem we have to check more:
First of all, since two general quintics intersect in 25 distinct points, it is clear that in step (3) of the construction we get for general choices
of $(P,R,|F|) \in X_2$ a collection $Q$ of eight distinct points disjoint from $P \cup R$.
What is not so clear but expected, is that for general choices
$h^0(\PP^2,\sI^2_P\cap \sI_Q(8))=1$, so that step (4) involves no further choices. This is easy to verify in our randomly chosen specific example computationally, and by semi-continuity we deduce that it holds as well  for general choices.
We  next check in our specific example that the curve $\Gamma$ has only 12 ordinary nodes and that $\Gamma$ and $R$ are disjoint. By semi-continuity  the same holds for general choices. So the pencil defines a degree $8$ map $C \to \PP^1$. We can check that it is simply ramified in our specific example, which by semi-continuity establishes this for general choices. This proves that we have 
a rational map
$$X_2 \dasharrow Y_2.$$
To prove that it is birational, we have to recover $P$ and $R$ from a triple $(\Gamma,L,Q) \in Y_2$. The set $P$ is the set of nodes.
To recover $R$ we check that in our specific randomly constructed example
$H^0(\PP^2,\sI_{P\cup Q}(5))=\langle f_1, f_2 \rangle$ is indeed a pencil of quintics without fix component, ie. $f_1,f_2$ have no common factor, and $25$ distinct base points,
so that we recover the homogenous ideal of $R$ as $I_R =(f_1, f_2):I_{P\cup Q}$. Since the condition $h^0(\PP^2,\sI_{P\cup Q}(5))=2$,  and the condition that this pencil has no fix component and $25$ distinct base points  are open conditions on triples $(\Gamma,L,Q) \in Y_2$, we get a rational inverse $Y_2 \dasharrow X_2$.
\end{proof}

\begin{corollary} $Y_2, Y_1=\sU_{9,8}\times_{\sM_9} \sW_{9,8}^1,  \sW_{9,8}^1$ and the Hurwitz scheme $\sH_{9,8}$ are unirational.
\end{corollary}

\begin{proof}
$X_2$ is unirational, since the rational map from the 17-fold product of $\PP^2$ to $X_2$
 is dominant.  $\sH_{9,8}$ as a $\PGL(2)$-bundle over $ \sW_{9,8}^1$ is unirational as well.
\end{proof}

\begin{remark} From the existence of an example with all desired properties over $\mathbb Q$ we can deduce the existence of example an with desired properties
over finite prime fields $\mathbb F_p$ for all but finitely many primes. Thus $\sH_{9,8}$ is also unirational over $\mathbb F_p$ for all but finitely many primes. Computing an explicit example over $\mathbb Q$ with integer coefficients allows to determine the possible exceptional primes, which then in principal could be checked case by case. We plan to incorporate this into a future update of our $\mathtt{Macaulay2}$ package \cite{DS}.
\end{remark}

\bigskip

\section*{Acknowledgement}
This work is done during the first author's visit at Universit\"at des Saarlandes. He thanks his colleagues there for their
warm hospitality. He also thanks Farhad Rahmati for his continuous support and valuable discussions. We thank the referee for the careful reading of the manuscript.

\end{document}